\documentclass[11pt,titlepage]{article}
\usepackage{graphicx}
\usepackage{amsmath}
\usepackage{amsfonts}
\usepackage{amsthm}
\usepackage{amssymb}
\usepackage{color}

\usepackage{verbatim} 
\usepackage{url}
\usepackage{hyperref}
\usepackage[margin=1in]{geometry}
\usepackage{setspace}
\usepackage{color}
\newtheorem{theorem}{Theorem}[section]
\newtheorem{lemma}[theorem]{Lemma}
\newtheorem{corollary}[theorem]{Corollary}

\newtheorem*{remark}{Remark}

\begin{document}

\title{Stern-Brocot Trees from Weighted Mediants}

\author{Dhroova Aiylam \\ MIT \\ \href{mailto:dhroova@mit.edu}{dhroova@mit.edu} \and Tanya Khovanova \\ MIT \\ \href{mailto:tanyakh@yahoo.com}{tanyakh@yahoo.com}}

\maketitle

\begin{abstract}
In this paper we discuss a natural generalization of the Stern Brocot tree which comes from the introduction of weighted mediants. We focus our attention on the case $k = 3$, in which $(2a + c)/(2b + d)$ and $(a + 2c)/(b + 2d)$ are the two mediants inserted between $a/b$ and $c/d$. Our main result is a determination of which rational numbers between the starting terms appear in the tree. We extend this result to arbitrary reduction schemes as well.

\end{abstract}

\section{Introduction}\label{sec:intro}

\indent

The Stern--Brocot tree is an object of classical interest in number theory. Discovered independently by Moritz Stern \cite{Stern} in 1858 and Achille Brocot \cite{Brocot} in 1861, it was originally used as a way to find rational approximations of certain kinds to specific numbers. As a consequence, the Stern-Brocot tree is deeply connected to the theory of continued fractions \cite{CTN}. It also comes up in a variety of other contexts, including Farey Sequences, Ford Circles, and Hurwitz' theorem \cite{CW, SDS, Reg, Gold}. 

The classical Stern-Brocot tree is generated row by row, as follows: the zeroth row has entries $\frac{0}{1}$ and $\frac{1}{0}$. In each subsequent row, all entries from the previous row are copied and between every pair of neighboring entries $\frac{a}{b}$ and $\frac{c}{d}$ the mediant fraction $\frac{a + c}{b + d}$ is inserted. This process is repeated ad infinitum; the result is the Stern--Brocot tree \cite{SB}.

The classical Stern-Brocot tree is well-understood, but there are several different variants that are natural candidates for study. For instance, one could consider varying the starting terms of the tree and ask which of the properties of the classical Stern-Brocot tree continue to hold, and to what extent. This question was addressed in detail in \cite{PRIMES}. The main result of that paper is a proof that regardless of the initial terms the Stern-Brocot tree contains every rational number between them. 

In Section~\ref{sec:defs}, we define precisely what is meant by the Stern-Brocot tree from weighted mediants given a pair of starting terms, an idea originally proposed by Prof. James Propp \cite{Propp}. We also define the cross-determinant and discuss its role in fraction reduction.

In Section~\ref{sec:numbersappear}, for an arbitrary Stern-Brocot tree we characterize which rational numbers between the starting terms appear. We turn this characterization into a simple, explicit description of these fractions. 

Finally, in Section~\ref{sec:reduction} we consider how non-uniform reduction of fractions impacts the Stern-Brocot tree and, in the process, introduce the idea of a reduction scheme. We expand our earlier result to deal with arbitrary reduction schemes.

\section{Weighted Mediants: Notation and Definitions}\label{sec:defs}

For a fixed parameter $k$, we say the weighted mediants of two fractions $a/b$, $c/d$ are 
$$\frac{(k - 1)a + c}{(k - 1)b + d}, \; \; \frac{(k - 2)a + 2c}{(k - 2)b + 2d}, \; \; \dots, \; \; \frac{a + (k - 1)c}{b + (k - 1)d}$$ whence there are $k - 1$ mediants in all. We stipulate that each of these fractions be reduced to lowest terms. As in the classical Stern-Brocot tree, the tree begins with two starting terms and each row is obtained by inserting mediants between consecutive fractions in the previous row.  With this notation, the classical Stern-Brocot tree is the case $k = 2$ with starting terms $0/1$ and $1/0$. The next row of this tree is $0/1$, $1/1$, and $1/0$. The two halves of the tree with respect to the mid-line are equivalent. Indeed if we swap numerators and denominators and reverse the order, the first part of the tree becomes the second part of the tree. For this reason, many researchers study only the first half of the tree.

In this paper we will restrict our attention to the case $k = 3$. While we treat the problem in fully general terms, one tree of particular interest to us is the one with starting terms $0/1$ and $1/1$. We call this $k = 3$ Stern-Brocot tree the \textit{unit tree}. Here is what the unit tree looks like:

	\[ \frac{0}{1} \; \; \; \frac{1}{1}\]  \[\frac{0}{1} \; \; \; \frac{1}{3} \; \; \; \frac{2}{3} \; \; \; \frac{1}{1}\] 
	\[\frac{0}{1} \; \; \; \frac{1}{5} \; \; \;  \frac{2}{7} \; \; \; \frac{1}{3} \; \; \; \frac{4}{9} \; \; \; \frac{5}{9} \; \; \; \frac{2}{3} \; \; \; \frac{5}{7} \; \; \; \frac{4}{5} \; \; \; \frac{1}{1}\]
	
It is easy to see that all the denominators in this tree must be odd. Later we will show that any number between 0 and 1 with an odd denominator in lowest terms appear in the tree.

Let us now introduce some notation and definitions. If $\frac{p}{q}$ and $\frac{r}{s}$ are rational numbers in lowest terms, their \emph{weighted mediants} are the numbers $\frac{2p + r}{2q + s}$ and $\frac{p + 2r}{q + 2s}$ in lowest terms. We call these the left and right mediants of $\frac{p}{q}$ and $\frac{r}{s}$, respectively.

Next let $SB(\frac{a}{b}, \frac{c}{d})$ stand for the ($k = 3$) Stern-Brocot tree with starting terms $\frac{a}{b}, \frac{c}{d}$, and denote the $i$-th row of this tree with $SB_i(\frac{a}{b}, \frac{c}{d})$. Thus $SB_0(\frac{a}{b}, \frac{c}{d}) = \{ \frac{a}{b}, \frac{c}{d}\}$ is the $0$-th row of the tree, and in general $SB_{i + 1}(\frac{a}{b}, \frac{c}{d})$ is obtained by copying all terms from $SB_{i}(\frac{a}{b}, \frac{c}{d})$ and inserting between every pair of consecutive fractions $\frac{p}{q}, \frac{r}{s} \in SB_{i}(\frac{a}{b}, \frac{c}{d})$ their weighted mediants. Thus the $i$-th row of the tree has $3^i+1$ numbers. As a matter of convention, we assume $\frac{a}{b} < \frac{c}{d}$ (of course, the reverse tree would simply be the reflection) and that $b, d \ge 0$. 

At first it might seem odd to permit $b, d = 0$, but recall the starting terms of the classical Stern-Brocot tree are $\frac{0}{1}$ and $\frac{1}{0}$. In the classical case,  $\frac{1}{0}$ is interpreted as $+ \infty$ and all the relevant conclusions (and their proofs) are the same. It stands to reason we should allow $b = 0$ or $d = 0$ --- not both; if $b=d=0$, all denominators of the fractions between them would be $0$ and it is a trivial case. 

We say the $\emph{cross-determinant}$ of two fractions $\frac{p}{q}$ and $\frac{r}{s}$ is $\mathcal{C}(\frac{p}{q}, \frac{r}{s}) = qr - ps$. We will be most interested in the cross-determinant of consecutive numbers in $SB_i(\frac{a}{b}, \frac{c}{d})$; as was the case when $k = 2$ \cite{PRIMES}, the cross-determinant essentially determines how fractions in the Stern-Brocot tree are capable of reducing. In particular, the factor by which the ratio of a weighted mediant is reduced to its lower terms is a factor of $\mathcal{C}(\frac{p}{q}, \frac{r}{s})$, as we will prove in Lemma~\ref{thm:cdetred}. We will see that the cross determinant of the starting terms is also important in determining which fractions will ultimately appear in the tree.

\begin{lemma}\label{thm:cdetred}
The factor by which a weighted mediant of two fractions $\frac{p}{q}$ and $\frac{r}{s}$ is reduced divides $\mathcal{C}(\frac{p}{q}, \frac{r}{s})$. 
\end{lemma}

\begin{proof}
Before reduction, the left mediant of $\frac{p}{q}$ and $\frac{r}{s}$ is $\frac{2p + r}{2q + s}$ and the right mediant  is $\frac{p + 2r}{q + 2s}$. Suppose the left mediant iss reduced by a factor $g$, so that after reduction it has numerator $\frac{2p + r}{g}$ and denominator $\frac{2q + s}{g}$. Then $$\mathcal{C} \left( \frac{p}{q}, \frac{\frac{2p + r}{g}}{\frac{2q + s}{g}} \right) = \frac{qr - ps}{g}.$$ Of course, $\mathcal{C}$ only takes integer values, so $g | qr - ps$. By analogous reasoning, the same is true for the right mediant. 
\end{proof}

The cross-determinant of two fractions is positive if the second fraction is larger than the first.  The cross-determinant is zero if and only if two fractions represent the same number.

It is important to remember that the cross-determinant depends on the representation of rational numbers, not just on the numbers themselves. The cross-determinant is the smallest when both rational numbers are in their lowest terms. 

Finally, given a rational number $x/y$ and an interval $I = [\frac{p}{q}, \frac{r}{s}]$, whose endpoints are ratios, we define the \textit{modulus} $m_I(x/y)$ of the number with respect to the interval's representation as the sum of cross-determinants with its end-points: 
$$m_I(x/y) = \mathcal{C}\left(\frac{p}{q}, \frac{x}{y}\right) + \mathcal{C}\left(\frac{x}{y}, \frac{r}{s}\right).$$

We usually consider the modulus only of $\frac{x}{y} \in [\frac{p}{q}, \frac{r}{s}]$, so that $m_I(x/y) > 0$. Notice that if $m_I(x/y) = 1$, then $x/y$ must coincide with one of the end points of the interval. As we will see in Section~\ref{sec:numbersappear}, the modulus is critical in the proof of which rationals appear in the Stern-Brocot tree. 

In Section~\ref{sec:numbersappear} we classify the numbers which appear in the Stern-Brocot tree.

\section{Rational Numbers in the Stern-Brocot tree}\label{sec:numbersappear}

Given a Stern-Brocot tree $SB(\frac{a}{b}, \frac{c}{d})$, imagine we want to determine whether some target rational $\frac{x}{y} \in [\frac{a}{b}, \frac{c}{d}]$ appears in the tree. Writing out the first few rows of any such tree makes it fairly clear that not all $x/y$ between the endpoints will appear. Indeed, from the first few rows of $SB(\frac{0}{1}, \frac{1}{1})$ it seems only fractions with odd denominator can ever appear in the tree. In fact, this is a special case of the following lemma:

\begin{lemma}\label{thm:modular}
Let $\frac{a}{b}$ and $\frac{c}{d}$ be fractions in lowest terms. All rational numbers $\frac{x}{y} \in SB(\frac{a}{b}, \frac{c}{d})$ satisfy either $(x, y) \equiv (a, b) \pmod{2}$ or $(x, y) \equiv (c, d) \pmod{2}$ where congruence is componentwise. Moreover, if $(a, b) \not\equiv (c, d) \pmod{2}$, then two consecutive fractions in the tree, alternate which of the parity equations they satisfy.
\end{lemma}

\begin{proof}
We prove by induction on $i$ that this holds for $SB_i(\frac{a}{b}, \frac{c}{d})$. Clearly the result holds when $i = 0$. Now suppose it holds when $i = n$, and consider $SB_{n + 1}(\frac{a}{b}, \frac{c}{d})$. All terms in this row were either also in $SB_{n}(\frac{a}{b}, \frac{c}{d})$, whence they satisfy the claim by the induction hypothesis, or the mediant of two consecutive terms $\frac{x}{y}, \frac{z}{w} \in SB_{n}(\frac{a}{b}, \frac{c}{d})$. The left mediant of these two fractions is the fraction $\frac{2x + z}{2y + w}$ in lowest terms. Notice that the numerator and denominator are not both even, since this would force $z$ and $w$ to both be even so that $\frac{z}{w}$ was not in lowest terms. Then whatever factor we reduce this fraction by must be odd, so that the parities of the numberator and denominator of the left mediant are $(2x + z, 2y + w) \equiv (z, w) \pmod{2}$. Yet by the induction hypothesis, $(z, w) \equiv (a, b) \pmod{2}$ or $(z, w) \equiv (c, d) \pmod{2}$. Thus the same is true for the left mediant of $\frac{x}{y}, \frac{z}{w}$, and by analogous reasoning, for the right mediant. The claim now follows by induction.
\end{proof}

\begin{corollary}\label{thm:parity}
If $(a, b) \not\equiv (c, d) \pmod{2}$, then $\mathcal{C}\left(\frac{a}{b}, \frac{c}{d}\right)$ is odd, and for any number $\frac{x}{y}$ in the tree one of the cross determinants $\mathcal{C}\left(\frac{a}{b}, \frac{x}{y}\right)$ or $\mathcal{C}\left(\frac{x}{y}, \frac{c}{d}\right)$ is odd.
\end{corollary}

For example, each row in the unit tree has fractions between $0/1$ and $1/1$ whose numerators alternate in terms of parity. Later, we will see that all the rational numbers with odd denominators in the range from 0 to 1 appear in the unit tree.

This is a good starting point; in many cases, such as the unit tree, the numbers which do not appear in the tree are precisely the ones forbidden by the lemma. Yet consider the tree $SB(\frac{1}{3}, \frac{3}{1})$: 

	\[ \frac{1}{3} \; \; \; \frac{3}{1}\]  \[\frac{1}{3} \; \; \; \frac{5}{7} \; \; \; \frac{7}{5} \; \; \; \frac{3}{1}\] 
	\[\frac{1}{3} \; \; \; \frac{7}{13} \; \; \;  \frac{11}{17} \; \; \; \frac{5}{7} \; \; \; \frac{17}{19} \; \; \; \frac{19}{17} \; \; \; \frac{7}{5} \; \; \; \frac{17}{11} \; \; \; \frac{13}{7} \; \; \; \frac{3}{1}\]
	
This tree has reciprocal symmetry about its midline; in particular, since each mediant operation produces two (distinct) fractions there is no way for the fraction $\frac{1}{1}$ to ever appear in this tree. Yet it is not ruled out by Lemma \ref{thm:modular}. To cover such cases as these, we need a refinement of Lemma \ref{thm:modular}. Let $\nu_{p}(n)$ denote the $p$-adic valuation of $n$.

\begin{lemma}\label{thm:2adic}
Let $\frac{a}{b}$ and $\frac{c}{d}$ be fractions in lowest terms. For all rational numbers $\frac{x}{y} \in SB(\frac{a}{b}, \frac{c}{d})$ 
$$\min \left(\nu_{2}\left(\mathcal{C}\left(\frac{a}{b}, \frac{x}{y}\right)\right),  \nu_{2}\left(\mathcal{C}\left(\frac{x}{y}, \frac{c}{d}\right)\right) \right) = \nu_{2}\left(\mathcal{C}\left(\frac{a}{b}, \frac{c}{d}\right)\right)\ \text{and}$$
$$\nu_{2}\left(\mathcal{C}\left(\frac{a}{b}, \frac{c}{d}\right)\right) < \max \left(\nu_{2}\left(\mathcal{C}\left(\frac{a}{b}, \frac{x}{y}\right)\right),  \nu_{2}\left(\mathcal{C}\left(\frac{x}{y}, \frac{c}{d}\right)\right) \right).$$ 
Moreover, if $\frac{p}{q}$ is an even-indexed fraction (where we start indexing with 0) in a row of the tree, then $\nu_{2} \left(\mathcal{C}\left(\frac{a}{b}, \frac{p}{q}\right)\right) > \nu_{2} \left(\mathcal{C}\left(\frac{p}{q}, \frac{c}{d}\right)\right)$. If instead $\frac{p}{q}$ has odd index, then $\nu_{2} \left(\mathcal{C}\left(\frac{a}{b}, \frac{p}{q}\right)\right) < \nu_{2} \left(\mathcal{C}\left(\frac{p}{q}, \frac{c}{d}\right)\right)$.
\end{lemma}

\begin{proof}
We prove by induction on $i$ that this holds for $SB_i(\frac{a}{b}, \frac{c}{d})$. Clearly the result holds when $i = 0$. Now suppose it holds when $i = n$, and consider $SB_{n + 1}(\frac{a}{b}, \frac{c}{d})$. All terms in this row are either also in $SB_{n}(\frac{a}{b}, \frac{c}{d})$, or else the mediant of two consecutive terms $\frac{p}{q}, \frac{r}{s} \in SB_{n}(\frac{a}{b}, \frac{c}{d})$.  In the first case, suppose the fraction occurs at index $I$ in $SB_{n}(\frac{a}{b}, \frac{c}{d})$; then in the next row $SB_{n + 1}(\frac{a}{b}, \frac{c}{d})$ it has index $3I \equiv I \pmod{2}$ whence the claim follows from the induction hypothesis. Now assume we are in the second case.

The left mediant of these two fractions is the fraction $\frac{2p + r}{2q + s}$ in lowest terms. As we saw in the proof of Lemma \ref{thm:modular}, the factor by which we reduce to lowest terms must be odd, and thus does not affect the $2$-adic valuation. Then we compute: 
$$\mathcal{C}\left(\frac{a}{b}, \frac{2p + r}{2q + s}\right) = 2(bp - aq) + (br - as) = 2\mathcal{C}\left(\frac{a}{b}, \frac{p}{q}\right) + \mathcal{C}\left(\frac{a}{b}, \frac{r}{s}\right)$$ 
and 
$$\mathcal{C}\left(\frac{2p + r}{2q + s}, \frac{c}{d}\right) = 2(cq - pd) + (cs - qd) = 2\mathcal{C}\left(\frac{p}{q}, \frac{c}{d}\right) + \mathcal{C}\left(\frac{r}{s}, \frac{c}{d}\right).$$ 

Suppose $\frac{p}{q}$ has even index in $SB_n(\frac{a}{b}, \frac{c}{d})$. Then $\nu_2 \left(\mathcal{C}(\frac{a}{b}, \frac{p}{q}) \right) > \nu_2\left(\mathcal{C}(\frac{a}{b}, \frac{c}{d})\right) = \nu_2\left(\mathcal{C}(\frac{a}{b}, \frac{r}{s}) \right)$ since $\frac{r}{s}$ has odd index ($\frac{p}{q}, \frac{r}{s}$ are consecutive). It follows that $\mathcal{C}\left(\frac{a}{b}, \frac{2p + r}{2q + s}\right) = \nu_2\left(\mathcal{C}(\frac{a}{b}, \frac{r}{s}) \right) =  \nu_2\left(\mathcal{C}(\frac{a}{b}, \frac{c}{d})\right)$. On the other hand, $\nu_2 \left(\mathcal{C}(\frac{p}{q}, \frac{c}{d}) \right) = \nu_2\left(\mathcal{C}(\frac{a}{b}, \frac{c}{d})\right) < \nu_2\left(\mathcal{C}(\frac{r}{s}, \frac{c}{d}) \right)$ so $\mathcal{C}\left(\frac{2p + r}{2q + s}, \frac{c}{d}\right)$ is at least $1 + \nu_2\left(\mathcal{C}(\frac{a}{b}, \frac{c}{d})\right) > \nu_2\left(\mathcal{C}(\frac{a}{b}, \frac{c}{d})\right)$. 

Instead if $\frac{p}{q}$ has odd index in $SB_n(\frac{a}{b}, \frac{c}{d})$, then $\nu_2 \left(\mathcal{C}(\frac{a}{b}, \frac{p}{q}) \right) = \nu_2\left(\mathcal{C}(\frac{a}{b}, \frac{c}{d})\right) < \nu_2\left(\mathcal{C}(\frac{a}{b}, \frac{r}{s}) \right)$ since $\frac{r}{s}$ has even index ($\frac{p}{q}, \frac{r}{s}$ are consecutive). It follows that $\mathcal{C}\left(\frac{a}{b}, \frac{2p + r}{2q + s}\right)$ is at least $1 + \nu_2\left(\mathcal{C}(\frac{a}{b}, \frac{p}{q}) \right) >  \nu_2\left(\mathcal{C}(\frac{a}{b}, \frac{c}{d})\right)$. On the other hand, $\nu_2 \left(\mathcal{C}(\frac{p}{q}, \frac{c}{d}) \right) > \nu_2\left(\mathcal{C}(\frac{a}{b}, \frac{c}{d})\right) = \nu_2\left(\mathcal{C}(\frac{r}{s}, \frac{c}{d}) \right)$ so $\mathcal{C}\left(\frac{2p + r}{2q + s}, \frac{c}{d}\right)$ is $ \nu_2\left(\mathcal{C}(\frac{r}{s}, \frac{c}{d})\right) = \nu_2\left(\mathcal{C}(\frac{a}{b}, \frac{c}{d})\right)$.

In either case, if $\frac{p}{q}$ has index $I$ in $SB_n(\frac{a}{b}, \frac{c}{d})$ then the left mediant of $\frac{p}{q}, \frac{r}{s}$ has index $3I + 1$ which is of opposite parity, and so the claim holds. The reasoning for the right mediant of $\frac{p}{q}, \frac{r}{s}$ is entirely analogous, and we are done by induction.
\end{proof}

Now let us characterize those numbers which appear in $SB(\frac{a}{b}, \frac{c}{d})$; we turn this into a precise description afterwards. Before we do so, we need a short technical lemma.

\begin{lemma}\label{thm:onethird}
The difference between the left and right mediants of $\frac{p}{q}$ and $\frac{r}{s}$ is at most one third of the difference between $\frac{p}{q}$ and $\frac{r}{s}$.
\end{lemma}

\begin{proof}
Of course, the possibility that the mediants are reduced is irrelevant; the value of the numbers is unchanged. Thus we compute $$\frac{p + 2r}{q + 2s} - \frac{2p + r}{2q + s} = \frac{3(qr - ps)}{(q + 2s)(2q + s)}.$$
Now $$\frac{r}{s} - \frac{p}{q} = \frac{qr - ps}{qs} = \frac{3(qr - qs)}{3qs}$$ whence it is enough to show that $$9qs \le (q + 2s)(2q + s) = 2q^2 + 2s^2 + 5qs \iff 2(q - s)^2 \ge 0$$ which is clear, and so we are done.
\end{proof}

We can now characterize numbers which appear in a particular tree $SB(\frac{a}{b}, \frac{c}{d})$. As we will see, it is actually more natural to characterize numbers which do \emph{not} appear in this tree. We have the following characterization:

\begin{theorem}\label{thm:numbersappear}
If a number $\frac{x}{y} \in [\frac{a}{b}, \frac{c}{d}]$ does not appear in $SB(\frac{a}{b}, \frac{c}{d})$, then $\frac{x}{y}$ is the mediant of two consecutive terms in some row of the tree.
\end{theorem}

\begin{remark}
The mediant we speak of here is the ordinary mediant which appears in the case $k = 2$; thus the mediant of $\frac{p}{q}$ and $\frac{r}{s}$ is $\frac{p + r}{q + s}$. 
\end{remark}

\begin{proof}
Assuming that $\frac{x}{y} \in [\frac{a}{b}, \frac{c}{d}]$ does not appear in the tree, we can find a sequence $\{I_n\}_{n \ge 0}$ with $I_0 = [\frac{a}{b}, \frac{c}{d}]$ and $I_n \supset I_{n+ 1}$ so that $x/y \in I_n$ and $I_n$ is the interval between two consecutive terms of the row $SB_{n}(\frac{a}{b}, \frac{c}{d})$. Once we have found the endpoints $\frac{a_n}{b_n}$, $\frac{c_n}{d_n}$ of $I_n$, the next row of the tree divides the interval $I_n = [a_n/b_n, c_n/d_n]$ into three sub-intervals, namely 
$$I_n = \left[\frac{a_n}{b_n}, \frac{2a_n + c_n}{2b_n + d_n}\right] \cup \left[\frac{2a_n + c_n}{2b_n + d_n}, \frac{a_n + 2c_n}{b_n + 2d_n}\right] \cup \left[\frac{a_n + 2c_n}{b_n + 2d_n}, \frac{c_n}{d_n}\right].$$

Now consider $m_{I_n}(\frac{x}{y})$. If $I_{n + 1}$ is the first segment: $[a_n/b_n, e/f]$, where $e/f = (2a_n + c_n)/(2b_n + d_n)$, or its reduced form, then 
$$m_{I_{n + 1}}\left(\frac{x}{y}\right) \leq (xb_n - ya_n) + (2a_n + c_n)y - (2b_n + d_n)x = m_{I_n}\left(\frac{x}{y}\right) - 2(xb_n - ya_n) < m_{I_n}\left(\frac{x}{y}\right).$$

The same is true if $I_{n + 1} = [(a_n + 2c_n)/(b_n + 2d_n), c_n/d_n]$ is the last segment (where the first number might be reduced). 

Finally if $I_{n + 1}$ is the middle interval, then 
$$m_{I_{n + 1}}\left(\frac{x}{y}\right) \leq  x(2b_n+d_n) - y(2a_n+c_n) + (a_n+2c_n)y - (b_n+2d_n)x = m_{I_{n}}\left(\frac{x}{y}\right).$$  

Equality holds only when there is no reduction. Thus $x/y$ falls into the middle interval, the endpoints of which are never reduced, all but finitely many times. Equivalently, there is some interval $I_N = [a_N/b_N, c_N/d_N] \ni x/y$ such that in every following row of the tree $x/y$ lies in the middle one of the three new intervals created and the endpoints are not reduced. These intervals are nested closed intervals with diameter tending to $0$ (by Lemma \ref{thm:onethird}), so their intersection is a single point. This point is $(a_N + c_N)/(b_N + d_N)$; indeed, the mediant of two fractions always lies between them, and $\frac{a_N + c_N}{b_N + d_N}$ is the mediant of every $I_n$, $n \ge N$. It follows that $x/y = (a_N + c_N)/(b_N + d_N)$ is the mediant of two consecutive terms in $SB(\frac{a}{b}, \frac{c}{d})$. 
\end{proof}

We can now turn this characterization into an explicit description of the fractions which appear in $SB(\frac{a}{b}, \frac{c}{d})$ using Lemmas \ref{thm:modular} and \ref{thm:2adic}. In particular, the criteria of Lemmas \ref{thm:modular} and \ref{thm:2adic} are both necessary and sufficient. 

\begin{theorem}\label{thm:containsallrational}
Let $\frac{a}{b}$ and $\frac{c}{d}$ be fractions in lowest terms. The Stern-Brocot tree $SB(\frac{a}{b}, \frac{c}{d})$ contains all rational numbers $\frac{x}{y}$ between $\frac{a}{b}$ and $\frac{c}{d}$ which satisfy the following conditions:

\begin{itemize}
\item $(x, y) \equiv (a, b) \pmod{2}$ or $(x, y) \equiv (c, d) \pmod{2}$
\item  $\min \left(\nu_{2}\left(\mathcal{C}\left(\frac{a}{b}, \frac{x}{y}\right)\right),  \nu_{2}\left(\mathcal{C}\left(\frac{x}{y}, \frac{c}{d}\right)\right) \right) = \nu_{2}\left(\mathcal{C}\left(\frac{a}{b}, \frac{c}{d}\right)\right) < \max \left(\nu_{2}\left(\mathcal{C}\left(\frac{a}{b}, \frac{x}{y}\right)\right),  \nu_{2}\left(\mathcal{C}\left(\frac{x}{y}, \frac{c}{d}\right)\right) \right)$
\end{itemize}

\end{theorem}

\begin{proof}
By Lemma \ref{thm:modular} and Lemma \ref{thm:2adic}, these conditions are necessarily satisfied by $\frac{x}{y} \in SB(\frac{a}{b}, \frac{c}{d})$. On the other hand, by Theorem~\ref{thm:numbersappear} we know the only rational numbers between $\frac{a}{b}$, $\frac{c}{d}$ that fail to appear in $SB(\frac{a}{b}$ and $\frac{c}{d})$ are the mediants of consecutive terms in the tree. 

Now consider the mediant $\frac{p + r}{q + s}$ of two consecutive terms $\frac{p}{q}$ and $\frac{r}{s}$ in $SB(\frac{a}{b}, \frac{c}{d})$. It is easy to see that unless this mediant is reduced by an even factor, it fails to meet the first condition of the theorem. Indeed, by Lemma \ref{thm:modular} we have $(p, q) \equiv (a, b)$ or $(p, q) \equiv (c, d)$ modulo $2$, and the same for $(r, s)$. As the equivalencies alternate $(p + r, q + s) \equiv (a + c, b + d) \pmod{2}$. The mediant fraction cannot satisfy the first condition of the theorem unless either $p$ and $q$ are both even or $r$ and $s$ are both even. Either way, this contradicts the fact that fractions have been reduced to lowest terms.

Note that $$\mathcal{C}\left(\frac{a}{b}, \frac{p + r}{q + s}\right)= \mathcal{C}\left(\frac{a}{b}, \frac{p}{q}\right) + \mathcal{C}\left(\frac{a}{b}, \frac{r}{s}\right).$$
Since, $\nu_2(\mathcal{C}(\frac{a}{b}, \frac{p}{q})) \neq \nu_2( \mathcal{C}(\frac{a}{b}, \frac{r}{s}))$ and $\min(\nu_2(\mathcal{C}(\frac{a}{b}, \frac{p}{q})),\nu_2( \mathcal{C}(\frac{a}{b}, \frac{r}{s}))) = \nu_2(\mathcal{C}(\frac{a}{b}, \frac{c}{d}))$, we have
$$\nu_{2}\left(\mathcal{C}\left(\frac{a}{b}, \frac{p + r}{q + s}\right)\right) = \nu_2\left(\mathcal{C}\left(\frac{a}{b}, \frac{c}{d}\right)\right).$$
Similarly, 
$$\nu_{2}\left(\mathcal{C}\left(\frac{p + r}{q + s},\frac{c}{d}\right)\right) = \nu_2\left(\mathcal{C}\left(\frac{a}{b}, \frac{c}{d}\right)\right).$$
Taking the reduction into account we see that 
$$\max \left( \nu_{2}\left(\mathcal{C}\left(\frac{a}{b}, \frac{p + r}{q + s}\right)\right), \nu_{2}\left(\mathcal{C}\left(\frac{p + r}{q + s}, \frac{c}{d}\right)\right)\right) < \nu_2\left(\mathcal{C}\left(\frac{a}{b}, \frac{c}{d}\right)\right).$$
The theorem follows.
\end{proof}

Consider for example the tree $SB(\frac{1}{3},\frac{3}{1})$ we discussed above. The cross-determinant of the initial terms is 8: $\mathcal{C}\left(\frac{1}{3}, \frac{3}{1}\right)=8$. It follows from the theorem that the numbers in the tree are all numbers $\frac{x}{y}$ in the range from $\frac{1}{3}$ to $\frac{3}{1}$ with odd numerators and denominators such that 8 divides $3x-y$ and $3y-x$.

As another example, consider the unit tree. We have already seen that numerators in each row alternate in parity and denominators are always odd. This means that the mediant of two consecutive numbers in a row has an odd numerator and an even denominator, and by Theorem \ref{thm:numbersappear} these are exactly the numbers that do not appear in the tree. It follows that any rational number between 0 and 1 with an odd denominator (in lowest terms) appears in the tree.

\section{Reduction}\label{sec:reduction}

So far we have stipulated, according to the definition of the tree, that all fractions should appear in lowest terms. This condition was important; for instance, the second row of  $SB(\frac{0}{1}, \frac{1}{1})$ contains the two consecutive entries $1/3$ and $4/9$. Their weighted mediants before reduction are $6/15$ and $9/21$, both of which are reducible by 3 and the new entries in the third row are $2/5$ and $3/7$. In this section, we consider what happens when we relax this assumption.

In the unit $k=2$ Stern-Brocot tree, fractions are always in lowest terms. Indeed, the cross-determinant is uniformly 1 in the unit $k = 2$ tree, and the factor by which fractions are reduced must divide their cross-determinants \cite{PRIMES}.

In the case $k = 3$, on the other hand, reduction is unavoidable. That is, regardless of the choice of starting terms there will be mediants which need to be reduced. To see why, consider the Stern-Brocot tree $SB(\frac{a}{b}, \frac{c}{d})$, and suppose no fractions were reduced in the first two rows. Then the fractions $\frac{2a + c}{2b + d}$ and $\frac{5a + 4c}{5b + 4d}$ appear consecutively in the second row, and their mediants $\frac{9a + 6c}{9b + 6d}$, $\frac{12a + 9c}{12b + 9d}$ are both reducible.

Let $R$ stand for a reduction scheme (a rule for how we reduce reducible fractions that appear in certain positions of the tree) and say the Stern-Brocot tree $SB(\frac{a}{b}, \frac{c}{d}, R)$ is generated exactly as the tree $SB(\frac{a}{b}, \frac{c}{d})$, except that reducible fractions are reduced according to $R$. We will use $R_u$ to represent uniform reduction to lowest term, so that all the above results are with respect to this reduction scheme. We also use $R_0$ to represent no reduction. These are the two most natural reduction schemes. Reduction schemes can in general be quite complex; for instance, we could flip a fair coin to decide whether or not to reduce a particular fraction, and make this choice independently for each fraction.

We now generalize Theorem \ref{thm:containsallrational} for a general reduction scheme. As we shall see, the proof proceeds almost identically once we critically examine which steps in the proof for $R = R_{u}$ depend, perhaps implicitly, on the reduction scheme and strengthen the necessary hypotheses. 

\begin{theorem}\label{thm:withred}
Let $\frac{a}{b}$ and $\frac{c}{d}$ be fractions (in lowest terms) and $R$ a reduction scheme. If neither $a \equiv b \equiv 0 \pmod{2}$ nor $c \equiv d \equiv 0 \pmod{2}$ (in particular, this is true if the starting terms are in lowest terms), then the Stern-Brocot tree $SB(\frac{a}{b}, \frac{c}{d}, R)$ contains a unique fraction representative of each one of the rational numbers $\frac{x}{y}$ (in lowest terms) between $\frac{a}{b}$ and $\frac{c}{d}$ which satisfy the following conditions:

\begin{itemize}
\item $(x, y) \equiv (a, b) \pmod{2}$ or $(x, y) \equiv (c, d) \pmod{2}$
\item  $\min \left(\nu_{2}\left(\mathcal{C}\left(\frac{a}{b}, \frac{x}{y}\right)\right),  \nu_{2}\left(\mathcal{C}\left(\frac{x}{y}, \frac{c}{d}\right)\right) \right) = \nu_{2}\left(\mathcal{C}\left(\frac{a}{b}, \frac{c}{d}\right)\right) < \max \left(\nu_{2}\left(\mathcal{C}\left(\frac{a}{b}, \frac{x}{y}\right)\right),  \nu_{2}\left(\mathcal{C}\left(\frac{x}{y}, \frac{c}{d}\right)\right) \right).$
\end{itemize}

\end{theorem}

\begin{proof}
First we see that Lemma \ref{thm:modular} is no longer true, since it relies on the fact that no fractions with even numerator and even denominator even appear. Luckily, there is a simple fix. If we inspect the proof we see that fractions with both even numerator and even denominator appear only when such fractions appeared in the previous row. Thus as long as we stipulate that neither $a \equiv b \equiv 0 \pmod{2}$ nor $c \equiv d \equiv 0 \pmod{2}$ where $\frac{a}{b}$ and $\frac{c}{d}$ denote the starting terms as usual, the proof of Lemma \ref{thm:modular} proceeds as before. 

Now that we are guaranteed that fractions cannot reduce by an even factor, the proof of Lemma \ref{thm:2adic} extends immediately to arbitrary $R$. Indeed, this is the only fact on which the proof depends. The technical Lemma \ref{thm:onethird} does not depend on reduction even implicitly, and so its proof goes unmodified.

Next we consider Theorem~\ref{thm:containsallrational}. Again from the fact that the modulus is non-decreasing we can conclude that unless a target rational $x/y$ falls into the middle interval, the endpoint of which are not reduced, all but finitely many times, it appears in the tree. Once more, the unique intersection point of these nested closed intervals is the (ordinary) mediant of two consecutive terms in $SB(\frac{a}{b}, \frac{c}{d}, R)$.

Finally, Theorem \ref{thm:containsallrational} is proved from Lemmas \ref{thm:modular} and \ref{thm:2adic} (which are both true with the appropriate strengthened hypothesis) and some basic analysis of $2$-adic valuations. By our assumption that fractions in $SB(\frac{a}{b}, \frac{c}{d}, R)$ do not reduce by an even factor, the analysis is unaffected. Thus we have the desired extension of Theorem \ref{thm:containsallrational}.
\end{proof}

We conclude with two remarks. First, notice Theorem \ref{thm:withred} applies even when $\frac{a}{b}$, $\frac{c}{d}$ are not in lowest terms, as long as they cannot be reduced by an even factor. On the other hand, if, for instance, $a \equiv b \equiv 0 \pmod{2}$, there is no reasonable way to classify the numbers which appear in the tree for general $R$. This is because we can carry fractions with even numerator and denominator as far through as we like, and then reduce them to obtain fractions which may be in a new parity class altogether.

For a specific example, let us consider the tree $SB(\frac{0}{2}, \frac{1}{1})$. In one reduction scheme, which we will call $R'$, new terms are reduced uniformly to lowest terms starting with the second row. The corresponding tree is:

	\[ \frac{0}{2} \; \; \; \frac{1}{1}\]  \[\frac{0}{2} \; \; \; \frac{1}{5} \; \; \; \frac{2}{4} \; \; \; \frac{1}{1}\] 
	\[\frac{0}{2} \; \; \; \frac{1}{9} \; \; \;  \frac{1}{6} \; \; \; \frac{1}{5} \; \; \; \frac{2}{7} \; \; \; \frac{5}{13} \; \; \; \frac{2}{4} \; \; \; \frac{5}{9} \; \; \; \frac{2}{3} \; \; \; \frac{1}{1}\]
	
Observe that this tree contains fractions which have even numerator and odd denominator. On the other hand, we can also consider this tree with the familiar reduction scheme $R_u$ for the new terms. In this case, $SB(\frac{0}{2}, \frac{1}{1}, R_u)$ is 

	\[ \frac{0}{2} \; \; \; \frac{1}{1}\]  \[\frac{0}{2} \; \; \; \frac{1}{5} \; \; \; \frac{1}{2} \; \; \; \frac{1}{1}\] 
	\[\frac{0}{2} \; \; \; \frac{1}{9} \; \; \;  \frac{1}{6} \; \; \; \frac{1}{5} \; \; \; \frac{1}{4} \; \; \; \frac{1}{3} \; \; \; \frac{1}{2} \; \; \; \frac{3}{5} \; \; \; \frac{3}{4} \; \; \; \frac{1}{1}\]
	
Since all fractions except $\frac{0}{2}$ have an odd numerator, all mediants will have an odd numerator before reduction (and therefore, after reduction) except potentially the right mediant of $\frac{0}{2}$ and its neighbor. Yet it is easy to see (by induction) that the fraction to the right of $\frac{0}{2}$ in row $k$ of the tree is $\frac{1}{4k + 1}$, whence the right mediant of $\frac{0}{2}$ and its neighbor is $\frac{2}{4k + 4}$ which becomes $\frac{1}{2k + 2}$ when reduced to lowest terms. Thus $SB(\frac{0}{2}, \frac{1}{1})$ contains fractions with even numerator and an odd denominator under some reduction schemes, but not others. So in some sense, Theorem 4.1 is the strongest possible statement.

\section{Acknowledgements}

We would like to thank James Propp for suggesting the project and discussing it with us.


\begin{thebibliography}{9}

\bibitem{PRIMES} D.~Aiylam, Modified Stern-Brocot Sequences, \url{http://arxiv.org/pdf/1301.6807v1.pdf}

\bibitem{Gold} M.~Benito, J. Javier Escribano, An Easy Proof of Hurwitz's Theorem, \textit{Amer. Math. Monthly}, Vol. 109, No. 10 (2002), pp. 916--918.

\bibitem{CTN} A.~Bogomolny, Stern-Brocot Tree, \url{http://www.cut-the-knot.org/blue/Stern.shtml}

\bibitem{Brocot} A.~Brocot, Calcul des rouages par approximation, nouvelle methode, \emph{Revue Chronometrique}, Vol. 3 (1861), 186--194.

\bibitem{CW} N.~Calkin, H.~S.~Wilf, Recounting the Rationals, \textit{Amer. Math. Monthly}, Vol. 107, No. 4 (2000), pp. 360--363.

\bibitem{SDS} D.~H.~Lehmer, On Stern's Diatomic Series, \textit{Amer. Math. Monthly}, Vol. 36, No. 2 (1929), pp. 59--67.

\bibitem{Propp} J.~Propp, Farey-ish fractions from weighted mediants, (2001) avalable at: \url{http://faculty.uml.edu/jpropp/CCCC-Apr2011.pdf}

\bibitem{Reg} B.~Reznick, Regularity properties of the Stern enumeration of the rationals, \textit{J. Integer Seq.}, Vol. 11 (2008), Article
08.4.1

\bibitem{Stern} M.~A.~Stern, Ueber eine zahlentheoretische Funktion, \emph{Journal fur die reine und angewandte Mathematik}, Vol. 55 (1858), 193--220.

\bibitem{SB} E.~W.~Weisstein, Stern-Brocot Tree, Wolfram Mathworld, \url{http://mathworld.wolfram.com/Stern-BrocotTree.html}




\end{thebibliography}
\end{document}